\documentclass[a4paper,oneside,10pt]{amsart}

\usepackage{a4wide}
\usepackage[T1]{fontenc}
\usepackage[ansinew]{inputenc}
\usepackage{lmodern} 
\usepackage{graphicx}
\usepackage{amsmath}
\allowdisplaybreaks
\usepackage{amsthm}
\usepackage{amsfonts}
\usepackage{amssymb}
\usepackage{setspace}
\usepackage{mathrsfs}
\usepackage[all]{xy}
\usepackage{enumerate}
\usepackage{xcolor}
\usepackage{catchfile}
\usepackage{autobreak}
\usepackage{todonotes}
\usepackage{lineno}
\theoremstyle{plain}

\newtheorem{theorem}{Theorem}[section]

\newtheorem{proposition}[theorem]{Proposition}
 \newtheorem{lemma}[theorem]{Lemma}

\theoremstyle{definition}
\newtheorem{remark}[theorem]{Remark}
\newtheorem{conjecture}[theorem]{Conjecture}

\newtheorem{assumptions}[theorem]{Assumption}

 \newtheorem{definition}[theorem]{Definition}

\newtheorem*{proposition*}{Proposition}
\newtheorem*{definition*}{Definition}
\numberwithin{equation}{section}
\theoremstyle{plain}
\newtheorem*{theorem*}{Theorem}

\newenvironment{abc}{\begin{enumerate}[{\rm (a)}]}{\end{enumerate}}

\def\dom{\mathrm{D}}

\def\dd{\mathrm{d}}

\def\Id{I} %{\mathrm{I}}

\def\RR{\mathbb{R}}

\def\NN{\mathbb{N}} 
\def\LLL{\mathscr{L}}

\def\D{\mathrm{D}}

\def\A{\mathcal{A}}
\def\L{\mathcal{L}}
\def\U{\mathcal{U}}
\def\D{\mathcal{D}}

\def\B{\mathcal{B}}
\def\Z{\mathcal{Z}}
\def\V{\mathcal{V}}
\def\Ell{\mathrm{L}}

\def\ddd#1{\mathop{\frac{\partial}{\partial{#1}}}}

\newcommand{\norm}[1]{\|#1\|}
\newcommand{\from}{\colon}



\begin{document}
\title{Perturbations of non-autonomous second-order abstract Cauchy problems} 

\author{Christian Budde}
\address{University of the Free State, Department of Mathematics and Applied Mathematics, P.O. Box 339, 9300 Bloemfontein, South Africa}
\email{buddecj@ufs.ac.za}

\author{Christian Seifert}
\address{Technische Universit\"at Hamburg, Institut f\"ur Mathematik, Am Schwarzenberg-Campus 3, 21073 Hamburg, Germany}
\email{christian.seifert@tuhh.de}

\begin{abstract}                                                                         
In this paper we present time-dependent perturbations of second-order non-au\-tonomous abstract Cauchy problems associated to a family of operators with constant domain. We make use of the equivalence to a first-order non-autonomous abstract Cauchy problem in a product space, which we elaborate in full detail. 
As an application we provide a perturbed non-autonomous wave equation.
\end{abstract}

\thanks{Declarations of interest: none.}

%\date{\today}
\maketitle

\section{Introduction}
Autonomous second-order abstract Cauchy problems, which are of the form
\begin{align}\label{eqn:IntroAut}\tag{ACP$_2$}
\begin{cases}
\ddot{u}(t) =Au(t),&\quad t>0,\\
u(0) =x,\\
\dot{u}(0) =y,
\end{cases}
\end{align}
for some (unbounded) operator $(A,\dom(A))$, which often occur in the context of wave equations, have been studied intensively by several authors in the past, e.g., Sova \cite{S1966}, Da Prato and Giusti \cite{DPG1967}, Fattorini \cite{F1985}, Neubrander \cite{N1986}, Xio and Jin \cite{XJ1990} as well as Xiao and Liang \cite{XL2003}. One can also find more information in the monographs by Arendt et al. \cite[Sect.~3.14 \& 3.15]{ABHN2011}, Melnikova and Filinkov \cite[Sect.~1.7]{MF2001} or Vasil'ev and Piskarev \cite{VP2004}. In contrast to the first-order problem, where (classical) solutions are given by $C_0$-semigroups, one needs another solution concept for \eqref{eqn:IntroAut}, the so-called cosine and sine families. Similar to the Hille--Yosida generation theorem for strongly continuous semigroups, one can also characterize generators of cosine families, cf. \cite[Thm.~3.15.3]{ABHN2011}, \cite[Thm.~1.7.2]{MF2001} or \cite[Thm.~A]{XJ1990}. The classical operator theoretical approach to \eqref{eqn:IntroAut} is to reduce these to first-order ones, where one can apply the theory of $C_0$-semigroups. For a detailed overview on $C_0$-semigroups, we refer for example to the monographs by Engel and Nagel \cite{EN}, Goldstein \cite{G2017} or Pazy \cite{P1983}. In \cite{Ki1972}, Kisy\'{n}ski gives an explicit correspondence between generators of cosine families and strongly continuous semigroups, see also \cite[Thm.~3.14.11]{ABHN2011}. %Besides the semigroup approach, one can also analyse the second order problems by direct methods, see for example \cite{FO1991,F1985,XL2003,XL1998}. However, we will use the reduction method in this article.

\medskip
In contrast to the autonomous second-order problems, one has the non-autonomous second-order abstract Cauchy problems of the form
\begin{align}\label{eqn:IntroNonAut}\tag{nACP$_2$}
\begin{cases}
\ddot{u}(t)=A(t)u(t),&\quad t\in\left(0,T\right],\\
u(0)=x,\\
\dot{u}(0)=y,
\end{cases}
\end{align}
for some fixed $T>0$, where $(A(t),\dom(A(t)))_{t\in\left[0,T\right]}$ is a family of operators. These non-autonomous second-order abstract Cauchy problems have been studied first by Kozak \cite{K1994,K1995,K1995Fund} and later on by Bochneck \cite{B1997}, Winiarska \cite{W2009,W2005} and Lan \cite{L2001}, just to mention a few. The same idea as for \eqref{eqn:IntroAut} helps to reduce \eqref{eqn:IntroNonAut} again to a first-order problem. Solutions of non-autonomous first-order abstract Cauchy problems have been studied exhaustively by means of evolution families for example by Acquistapace and Terreni \cite{Acquistapace1987} and Kato and Tanabe \cite{Tanabe1960,Kato1961,tanabeBook}. A semigroup approach by so-called evolution semigroups was firstly introduced by Howland \cite{H1974} and later on studied by several authors, e.g., Evans \cite{E1976}, Nagel \cite{Na1995}, Nickel \cite{N1997}, Rhandi \cite{NR1995} and Schnaubelt \cite{RS1996}.

\medskip
The goal of this paper is to establish a perturbation result for \eqref{eqn:IntroNonAut}. Perturbation theorems for cosine families associated to \eqref{eqn:IntroAut} have been developed for example by Piskarev and Shaw \cite{PS1995,PS1997}, Miyadera \cite{SM1978}, Takenaka and Okazawa \cite{TO1978} as well as Travis and Webb \cite{TW1981}. Also time-dependent perturbations have been studied, cf. \cite{SW1986,L1998,Lu1981}. We want to perturb \eqref{eqn:IntroNonAut} in a time-dependent way as it has been done for first-order non-autonomous problems by R\"abinger et al. \cite{RRS1996,RSRV2000}. Especially, we want to cover time-dependent perturbations of bounded type.

\medskip

The paper is structured as follows. Section \ref{sec:preliminaries} consists of preliminary definitions regarding solutions of non-autonomous abstract Cauchy problems. 
The following Section \ref{sec:generation} provides a relation between existence of fundamental solutions to first- and second-order non-autonomous Cauchy problems.
This can be viewed as a non-autonomous version of the generation theorem of Kisy\'{n}ski in \cite{Ki1972}, cf.\ Theorem \ref{thm:Gen}. Section \ref{sec:BoundPert} provides our main result regarding perturbations of second-order non-autonomous Cauchy problems, cf.\ Theorem \ref{thm:bounded_perturbation}. The final Section \ref{sec:example} provides an example of a non-autonomous wave equation.

\section{Preliminaries on Non-Autonomous Abstract Cauchy Problems}
\label{sec:preliminaries}

Let $X$ be a Banach space, $T>0$ and $(A(t),\dom(A(t)))_{t\in\left[0,T\right]}$ be a family of closed operators on $X$.

We make the following crucial assumption throughout this article.

\begin{assumptions}\label{assump:ConstDom}
The domains $\dom(A(t))$ do not depend on time, i.e., there exists $D\subseteq X$ such that $\dom(A(t)) = D$ for all $t\in [0,T]$.
\end{assumptions}

We write $\Delta:=\Delta_T:=\{(t,s)\in\left[0,T\right]^2:\ t\geq s\}$.

\subsection{Non-Autonomous First-Order Cauchy Problems}

Non-autonomous first-order abstract Cauchy problems are of the form
\begin{align}\tag{nACP}\label{eqn:nACP}
\begin{cases}
\dot{v}(t)=A(t)v(t),&\quad t\in\left(0,T\right],\\
v(0)=x,
\end{cases}
\end{align}
where $x\in X$.

\begin{definition}\label{def:SolnACP}
A function $v\from\left[0,T\right]\to X$ is called a \emph{(classical) solution} to \eqref{eqn:nACP} if $v$ is continuously differentiable, $v(t)\in D$ for all $t\in\left(0,T\right]$ and $v$ satisfies \eqref{eqn:nACP}.
\end{definition}

The following solution concept is important for non-autonomous first-order problems. 

\begin{definition}\label{def:FundSolnACP}
A \emph{fundamental solution} to \eqref{eqn:nACP} associated with $(A(t),\dom(A(t)))_{t\in\left[0,T\right]}$ satisfying Assumption \ref{assump:ConstDom} is a family of bounded linear operators $(U(t,s))_{(t,s)\in\Delta}$ on a Banach space $X$ satisfying the following conditions:
\begin{itemize}
	\item[\textbf{(U1)}]\label{item:U1} $U(t,t)=\Id$ and $U(t,s)U(s,r)=U(t,r)$ for all $(t,s),(s,r)\in\Delta$.
	\item[\textbf{(U2)}]\label{item:U2} The mapping $\Delta\ni(t,s)\mapsto U(t,s)$ is strongly continuous on $X$.
	\item[\textbf{(U3)}]\label{item:U3} $U(t,s)D\subseteq D$
	\item[\textbf{(U4)}]\label{item:U4} For all $(t,s)\in\Delta$ and $x\in D$ one has that $\frac{\partial}{\partial{t}}U(t,s)x$ and $\frac{\partial}{\partial{s}}U(t,s)x$ exist and
	\begin{align}\label{eqn:DiffEvoFam}
	\frac{\partial}{\partial{t}}U(t,s)x=A(t)U(t,s)x \text{ and } \frac{\partial}{\partial{s}}U(t,s)x=-U(t,s)A(s)x.
	\end{align}
\end{itemize}
\end{definition} 
Note that \textbf{(U4)} requires \textbf{(U3)}.

\begin{remark}
\begin{abc}
  \item Observe that a fundamental solution to \eqref{eqn:nACP} is bounded in $\L(X)$, which follows from \textbf{(U2)}, compactness of $\Delta$ and the uniform boundedness principle.
	\item Fundamental solutions of \eqref{eqn:nACP} as defined in Definition \ref{def:FundSolnACP} are also often called evolution families.
\end{abc}	
\end{remark}

\begin{proposition}
Let $(U(t,s))_{(t,s)\in\Delta}$ be a fundamental solution of \eqref{eqn:nACP}. If $x\in D$ then $v(t):=U(t,0)x$ defines a classical solution of \eqref{eqn:nACP}.
\end{proposition}

\begin{proof}
    The statement follows easily from the properties \textbf{(U1)}--\textbf{(U4)}.
\end{proof}

\subsection{Non-Autonomous Second-Order Cauchy Problems}

Non-autonomous second-order abstract Cauchy problems are of the form
\begin{align}\label{eqn:nACP2}\tag{nACP$_2$}
\begin{cases}
\ddot{u}(t)=A(t)u(t),&\quad t\in\left(0,T\right],\\
u(0)=x,\\
\dot{u}(0)=y,
\end{cases}
\end{align}
where $x,y\in X$.

\begin{definition}\label{def:SolnACP2}
A function $u\from \left[0,T\right]\to X$ is called a \emph{(classical) solution} to \eqref{eqn:nACP2} if $u$ is twice continuously differentiable, $u(t)\in D$ for all $t\in\left(0,T\right]$ and $u$ satisfies \eqref{eqn:nACP2}.
\end{definition}

The solution concept of \eqref{eqn:nACP2} is vastly more involved than this of \eqref{eqn:nACP}. In fact, it goes back to Kozak \cite[Def.~2.1, Def.~3.1]{K1995Fund}. Note that our definition is slightly different from the one of Kozak, see also Remark \ref{rem:different_definitions} below.

\begin{definition}\label{def:FundSolnACP2}
A \emph{fundamental solution} to \eqref{eqn:nACP2} associated with $(A(t),\dom(A(t)))_{t\in\left[0,T\right]}$ satisfying Assumption \ref{assump:ConstDom} is a family of bounded linear operators $(S(t,s))_{(t,s)\in\Delta}$ on $X$ satisfying the following conditions: 
\begin{itemize}
	\item[\textbf{(S1)}]\label{item:S1} 
	\begin{abc}
		\item $S(t,t)=0$ for all $t\in\left[0,T\right]$.
		\item The mapping $\Delta\ni (t,s)\mapsto S(t,s)$ in strongly continuous on $X$.
		\item 
		For all $x\in X$ and $s\in [0,T]$ the mapping $[s,T]\ni t\mapsto S(t,s)x$ is continuously differentiable, and $(t,s)\mapsto \ddd{t}S(t,s)x$ is continuous with
		\[\left.\frac{\partial}{\partial{t}}S(t,s)x \right|_{t=s} =x.\]
		\item 
		For all $x\in D$ and $t\in [0,T]$ the mapping $[0,t]\ni s\mapsto S(t,s)x$ is continuously differentiable, and $(t,s)\mapsto \ddd{s}S(t,s)x$ is continuous with
		\[\left.\frac{\partial}{\partial{s}}S(t,s)x \right|_{t=s}=-x.\]
	\end{abc}
	\item[\textbf{(S2)}]\label{item:S2} $S(t,s)D\subseteq D$ for all $(t,s)\in\Delta$, for $x\in D$ the mapping $\Delta\ni(t,s)\mapsto S(t,s)x$ is twice continuously differentiable and
	\begin{abc}
		\item $\displaystyle{\frac{\partial^2}{\partial{t^2}}S(t,s)x=A(t)S(t,s)x.}$
		\item $\displaystyle{\frac{\partial^2}{\partial{s^2}}S(t,s)x=S(t,s)A(s)x.}$
		\item $\displaystyle{\left.\frac{\partial}{\partial{t}}\frac{\partial}{\partial{s}}S(t,s)x \right|_{t=s}=0}$
	\end{abc}
	\item[\textbf{(S3)}]\label{item:S3} For all $(t,s)\in\Delta$, if $x\in D$ then $\frac{\partial}{\partial s}S(t,s)x\in D$, there exist $\frac{\partial^2}{\partial{t^2}}\frac{\partial}{\partial{s}}S(t,s)x$ and $\frac{\partial^2}{\partial{s^2}}\frac{\partial}{\partial{t}}S(t,s)x$ and the following properties hold
	\begin{abc}
		\item $\displaystyle{\frac{\partial^2}{\partial{t^2}}\frac{\partial}{\partial{s}}S(t,s)x=A(t)\frac{\partial}{\partial{s}}S(t,s)x.}$
		\item $\displaystyle{\frac{\partial^2}{\partial{s^2}}\frac{\partial}{\partial{t}}S(t,s)x=\frac{\partial}{\partial{t}}S(t,s)A(s)x.}$
		\item The mapping $\Delta\ni(t,s)\mapsto A(t)\frac{\partial}{\partial{s}}S(t,s)x$ is continuous.
	\end{abc}
\end{itemize}
Moreover, we call a fundamental solution $(S(t,s))_{(t,s)\in\Delta}$ \emph{evolutionary} if additionally
\begin{itemize}
	\item[\textbf{(S4)}]\label{item:S4} For all $(t,s),(s,r)\in\Delta$ and $x\in D$ one has 
	\begin{align*}
	\left(-\frac{\partial}{\partial{s}}S(t,s)\right)S(s,r)x+S(t,s)\frac{\partial}{\partial{s}}S(s,r)x=S(t,r)x.
	\end{align*}
\end{itemize}

\end{definition}

\begin{remark}
\label{rem:different_definitions}
\begin{abc}
	\item Note that our Definition of fundamental solutions is slightly different from \cite[Def.~2.1, Def.~3.1]{K1995Fund} in the sense that we do not assume $(t,s)\mapsto S(t,s)x$ to be continuously differentiable for all $x\in X$ in \textbf{(S1)}, but allow for the partial derivative $\ddd{s}S(t,s)x$ only for $x\in D$ in \textbf{(S1)}(d). Thus, our definition is also different from the ones used in \cite{B1997,H2011,H2013,HP2016}. The reason we adjusted the definition stems from \cite[Theorem 4.1, 5)]{K1995} which seems to be only possible for $x\in \dom(B)$ instead of $x\in X$ there. In fact, the reasoning in \cite[Remark 4.2]{K1995} does not apply for $\ddd{s}S(t,s)$ in the topology of $X$, but only in the topology of $\dom(B)$ there.
	\item Moreover, in contrast to \cite{K1995,K1995Fund,B1997,H2011,H2013,HP2016}, we only define the fundamental solutions on $\Delta$ instead of $[0,T]^2$, see also \cite{K1994}.
\end{abc}	
\end{remark}

\begin{lemma}
\label{lem:properties_S}
    Let $(S(t,s))_{(t,s)\in\Delta}$ be a fundamental solution of \eqref{eqn:nACP2} in $X$. Then $(S(t,s))_{(t,s)\in\Delta}$ and $(\frac{\partial}{\partial t} S(t,s))_{(t,s)\in\Delta}$ are bounded in $\L(X)$.
\end{lemma}

\begin{proof}
    Since \textbf{(S1)}(b)--(c) yield that $(t,s)\mapsto S(t,s)x$ and $(t,s)\mapsto \ddd{t}S(t,s)x$ are continuous on the compact set $\Delta$ for all $x\in X$, we observe that $(S(t,s))_{(t,s)\in\Delta}$ and $(\frac{\partial}{\partial t} S(t,s))_{(t,s)\in\Delta}$ 
    are pointwise bounded. The uniform boundedness principle then yields boundedness in $\L(X)$.
\end{proof}

\begin{proposition}
\label{prop:solution_nACP2}
Let $(S(t,s))_{(t,s)\in\Delta}$ be a fundamental solution of \eqref{eqn:nACP2}. If $x,y\in D$, then $u(t):=- \ddd{s}S(t,0)x + S(t,0)y$ defines a classical solution of \eqref{eqn:nACP2}.
\end{proposition}

\begin{proof}
    The statement follows easily from the properties \textbf{(S1)}--\textbf{(S3)}.
\end{proof}

\section{Existence of Fundamental Solutions: A Generation Type Result}
\label{sec:generation}

Let $X$ be a Banach space, $T>0$, $(A(t),\dom(A(t)))_{t\in\left[0,T\right]}$ a family of closed linear operators in $X$ satisfying Assumption \ref{assump:ConstDom} with $D\subseteq X$ dense. We will further assume the following.

\begin{assumptions}\label{ass:EquiNorm}
There exists $C>0$ such that
\[
\frac{1}{C} \norm{\cdot}_{A(s)} \leq \norm{\cdot}_{A(t)} \leq  C \norm{\cdot}_{A(s)} \quad(s,t\in [0,T]).
\]
We equip $D$ with the graph norm of $A(0)$ (equivalently, with the graph norm of any $A(t)$) such that $D$ is a Banach space.
\end{assumptions}

It is worth to mention, that the equivalence of the graph norms is an assumption made regularly throughout the literature, see for example \cite[Rem.~4.5]{S2004}, \cite[Rem.~4.2]{VZ2008} and \cite[Sect.~7]{Ama1988}, just to mention a few.

\medskip
Moreover, the following regularity assumption on $A(\cdot)$ will be made.

\begin{assumptions}\label{ass:regularity_A}
    Assume $A(\cdot)x\in\mathrm{C}^1([0,T];X)$ for all $x\in D$.
\end{assumptions}

Note that the first-order non-autonomous abstract Cauchy problem \eqref{eqn:nACP} is well-posed according to \cite[Chapter VI, Def.~9.1]{EN} if Assumption \ref{assump:ConstDom} and Assumption \ref{ass:regularity_A} are satisfied and every operator generates a contractive $C_0$-semigroup, cf. \cite{K1953,K1970}.

\medskip
The following lemma will be useful in what follows.

\begin{lemma}
\label{lem:continuity_interpolationspaces}
Let $Z,X$ be Banach spaces, $Z\subseteq X$, $B\in\L(X)$.
\begin{abc}
 \item Assume $BX\subseteq Z$. Then $B\in \L(X,Z)$.
 \item Assume $BZ\subseteq Z$. Then $B\in \L(Z)$.
\end{abc}
\end{lemma}

\begin{proof}
\begin{abc}
 \item Let $(x_n)$ in $X$, $x\in X$, $y\in Z$ such that $x_n\to x$ in $X$ and $Bx_n\to y$ in $Z$. Since $Z\subseteq X$ we also have $Bx_n\to y$ in $X$. Hence, by continuity of $B$ on $X$, we have $Bx=y$, so $B\colon X\to Z$ is closed and hence bounded by the closed graph theorem.
 \item Let $(x_n)$ in $Z$, $x\in Z$, $y\in Z$ such that $x_n\to x$ in $Z$ and $Bx_n\to y$ in $Z$. Since $Z\subseteq X$ we also have $x_n\to x$ in $X$ and $Bx_n\to y$ in $X$. Hence, by continuity of $B$ on $X$, we have $Bx=y$, so $B\colon Z\to Z$ is closed and hence bounded by the closed graph theorem.
\end{abc}    
\end{proof}

Let $Z$ be a Banach space such that $D\subseteq Z$ dense and $Z\subseteq X$ dense.
Let
$(\A(t),\dom(\A(t)))_{t\in\left[0,T\right]}$ in $\Z:=Z\times X$ be defined by
\begin{align*}
\A(t):=\begin{pmatrix}
0&\Id\\
A(t)&0
\end{pmatrix},\quad \dom(\A(t)):=\D:=D\times Z.
\end{align*}

\begin{theorem}\label{thm:Gen}

The following assertions are equivalent:
\begin{abc}
	\item\label{thm:Gen:item:nACP2} There exists an evolutionary fundamental solution $(S(t,s))_{(t,s)\in\Delta}$ on $X$ of \eqref{eqn:nACP2} associated to $(A(t),D)_{t\in\left[0,T\right]}$ such that
	for all $(t,s)\in\Delta$ we have 
	\begin{itemize}
	\item $S(t,s)X\subseteq Z$, $S(t,s)Z\subseteq D$, $(t,s)\mapsto S(t,s)x\in Z$ is continuous for all $x\in X$, 
	\item $\ddd{t}S(t,s)Z\subseteq Z$, $\frac{\partial^2}{\partial t^2}S(t,s)x$ exists for all $x\in Z$ and $\frac{\partial^2}{\partial t^2}S(t,s)x = A(t)S(t,s)x$, 
	\item $\ddd{s} S(t,s)x$ exists for all $x\in Z$, $\ddd{s} S(t,s)Z\subseteq Z$ and $(t,s)\mapsto \ddd{s}S(t,s)x\in Z$ is continuous for all $x\in Z$, 
	\item $\ddd{t}\ddd{s}S(t,s)D\subseteq Z$, $\ddd{t}\ddd{s} S(t,s)x$ exists for all $x\in Z$ and there exists $C\geq 0$ such that $\|\ddd{t}\ddd{s} S(t,s)x\|_X\leq C\|x\|_{Z}$ for all $x\in Z$ and $(t,s)\in\Delta$.
	\end{itemize}
	\item\label{thm:Gen:item:nACP} There exists a fundamental solution $(\U(t,s))_{(t,s)\in\Delta}$ on $\L(\Z)$ of \eqref{eqn:nACP} associated to $(\A(t),\D)_{t\in\left[0,T\right]}$.
\end{abc}
\end{theorem}

\begin{proof}
We first prove the implication \eqref{thm:Gen:item:nACP}$\Longrightarrow$\eqref{thm:Gen:item:nACP2}.

\medskip

Define $(S(t,s))_{(t,s)\in\Delta}$ on $X$ by
\begin{align}\label{eqn:FundSolnACP2}
S(t,s)x:=\pi_1\U(t,s)\begin{pmatrix} 0\\x\end{pmatrix},\quad x\in X,
\end{align} 
where $\pi_1\from Z\times X \to Z$ denotes the projection on the first component. Hence, $S(t,s)\in \L(X,Z)$ and since $Z\subseteq X$ we also have $S(t,s)\in \L(X)$ for all $(t,s)\in\Delta$. Since $(\U(t,s))_{(t,s)\in\Delta}$ is strongly continuous on $\Z$, also $(t,s)\mapsto S(t,s) \in \L(X,Z)$ strongly continuous.

\medskip
Let us provide some useful formulas. First, we note that $\pi_1 \A(t) = \pi_2|_{\D}$, where $\pi_2\from Z\times X\to X$ is the projection on the second component. Second, for $x\in Z$ we have that $\A(s)\begin{pmatrix}0\\x\end{pmatrix} = \begin{pmatrix}x\\0\end{pmatrix}$ for all $s\in[0,T]$. In particular, for $x\in D$ we have $\A(s) \begin{pmatrix}0\\x\end{pmatrix} \in \D$ for all $s\in[0,T]$.

\medskip
For $x\in Z$ we have $\begin{pmatrix}0\\x\end{pmatrix}\in \D$, so by \textbf{(U3)} we obtain $S(t,s)x\in D$ for all $(t,s)\in \Delta$. Moreover, $t\mapsto S(t,s)x$ is differentiable and 
\[\ddd{t}S(t,s)x = \pi_1 \ddd{t}\U(t,s)\begin{pmatrix}0\\x\end{pmatrix} = \pi_1 \A(t)\U(t,s)\begin{pmatrix}0\\x\end{pmatrix} = \pi_2 \U(t,s)\begin{pmatrix}0\\x\end{pmatrix}.\]
By \textbf{(U3)} we observe $\ddd{t}S(t,s)x \in Z$.

\medskip
Let $x\in Z$. Then, by \textbf{(U4)}, we have that $\U(t,s)\begin{pmatrix} 0\\x\end{pmatrix}$ is differentiable with respect to $s$ and we obtain
\[\ddd{s}\U(t,s)\begin{pmatrix} 0\\x\end{pmatrix} = -\U(t,s)\A(s)\begin{pmatrix} 0\\x\end{pmatrix} = -\U(t,s) \begin{pmatrix} x\\0\end{pmatrix}.\]
Thus, $\ddd{s}S(t,s)x$ exists, $\ddd{s}S(t,s)x\in Z$ and
\[\ddd{s}S(t,s)x = \pi_1 \ddd{s}\U(t,s)\begin{pmatrix} 0\\x\end{pmatrix} = -\pi_1 \U(t,s) \begin{pmatrix} x\\0\end{pmatrix}\]
for all $(t,s)\in\Delta$. As $(\U(t,s))_{(t,s)\in\Delta}$ is strongly continuous, also $(t,s)\mapsto \ddd{s}S(t,s)x\in Z$ is continuous for all $x\in Z$.

\medskip
We now show existence of the mixed derivative of $(S(t,s))_{(t,s)\in\Delta}$. First, let $x\in D$. 
Then, by \textbf{(U4)} and the above, we have that $\ddd{s}\U(t,s)\begin{pmatrix} 0\\x\end{pmatrix}$ is differentiable with respect to $t$ and
\[\ddd{t}\ddd{s}\U(t,s)\begin{pmatrix} 0\\x\end{pmatrix} = -\ddd{t}\U(t,s) \begin{pmatrix} x\\0\end{pmatrix} = -\A(t)\U(t,s)\begin{pmatrix}x\\0\end{pmatrix}.\]
Thus, $\ddd{t}\ddd{s}S(t,s)x$ exists and
\[\ddd{t}\ddd{s}S(t,s)x = -\pi_1 \A(t)\U(t,s)\begin{pmatrix}x\\0\end{pmatrix} = \pi_2 \U(t,s)\begin{pmatrix}x\\0\end{pmatrix}.\]
By \textbf{(U3)} this shows that $\ddd{t}\ddd{s}S(t,s)x\in Z$.

\medskip
Let $x\in Z$. Then $\begin{pmatrix} 0\\x\end{pmatrix}\in \D$, so as above
\[\ddd{t} S(t,s)x = \pi_2 \U(t,s)\begin{pmatrix} 0\\x\end{pmatrix}.\]
As the right-hand side is differentiable with respect to $t$, $\frac{\partial^2}{\partial t^2} S(t,s)x$ exists and
\begin{align*}
 \frac{\partial^2}{\partial t^2} S(t,s)x & = \pi_2 \A(t)\U(t,s)\begin{pmatrix} 0\\x\end{pmatrix} = \pi_2 \A(t)\begin{pmatrix} \pi_1\\\pi_2\end{pmatrix} \U(t,s)\begin{pmatrix} 0\\x\end{pmatrix} \\
 & = A(t)\pi_1 \U(t,s)\begin{pmatrix} 0\\x\end{pmatrix} = A(t) S(t,s)x.
\end{align*}

Now, since $(\U(t,s))_{(t,s)\in \Delta}$ is bounded on $\Z$, there exists $C\geq 0$ such that $\norm{\U(t,s)}_{\L(\Z)}\leq C$ for all $(t,s)\in\Delta$. 
Thus,
\begin{equation}
\label{eqn:dtdsS_Z_X}
\norm{\ddd{t}\ddd{s}S(t,s)x}_X \leq C\norm{x}_{Z}\end{equation}
for all $(t,s)\in\Delta$ and $x\in Z$. 
Let $x\in Z$ and choose $(x_n)$ in $D$ such that $x_n\to x$ in $Z$. Then $t\mapsto \ddd{s}S(t,s)x_n$ is continuously differentiable for all $n\in\NN$ and these functions converge uniformly to $t\mapsto \ddd{s}S(t,s)x$, by boundedness of $(\U(t,s))_{(t,s)\in\Delta}$. Moreover, $t\mapsto \ddd{t}\ddd{s}S(t,s)x_n$ is continuous  for all $n\in\NN$ and these functions converge uniformly by the estimate above. Thus, $t\mapsto \ddd{s}S(t,s)x$ is continuously differentiable and \eqref{eqn:dtdsS_Z_X} holds for all $x\in Z$.

\medskip

We are now ready to prove \textbf{(S1)}--\textbf{(S4)} similarly as in \cite[Sect.~3]{K1994} and \cite[Thm.~4.1]{K1995}.

\medskip

\textbf{(S1)}: We prove (a)--(d) step by step.
\begin{abc}
 \item
 For $t\in[0,T]$ we have $\U(t,t) = \Id$ and therefore $S(t,t) = 0$. 
 \item 
 Strong continuity of $(t,s)\mapsto S(t,s)$ follows from strong continuity of $(t,s)\mapsto \U(t,s)$ and continuity of $\pi_1$.
 \item
 Let $s\in [0,T]$ and $x\in Z$. Then $\begin{pmatrix}0\\x\end{pmatrix}\in \D$ and \textbf{(U4)} yields that $[s,T]\ni t \mapsto \U(t,s)\begin{pmatrix}0\\x\end{pmatrix}$ is continuously differentiable. Thus, also $[s,T]\ni t\mapsto S(t,s)x$ is continuously differentiable, and by \textbf{(U4)} we compute
\[\frac{\partial}{\partial{t}}S(t,s)x = \frac{\partial}{\partial{t}} \pi_1\U(t,s)\begin{pmatrix} 0\\x\end{pmatrix} = \pi_1 \frac{\partial}{\partial{t}}\U(t,s)\begin{pmatrix} 0\\x\end{pmatrix} = \pi_1 \A(t)\U(t,s)\begin{pmatrix} 0\\x\end{pmatrix} = \pi_2 \U(t,s)\begin{pmatrix} 0\\x\end{pmatrix}.\]
Now, for $x\in X$ we choose $(x_n)$ in $Z$ such that $x_n\to x$ in $X$. Then $t\mapsto S(t,s)x_n$ is continuously differentiable for all $n\in\NN$ and these functions converge uniformly to $t\mapsto S(t,s)x$, by boundedness of $(S(t,s))_{(t,s)\in\Delta}$. Moreover, $t\mapsto \ddd{t}S(t,s)x_n$ is continuous for all $n\in\NN$ and these functions converge uniformly by boundedness of $(\U(t,s))_{(t,s)\in\Delta}$. Thus, $t\mapsto S(t,s)x$ is continuously differentiable with
\[\frac{\partial}{\partial{t}}S(t,s)x = \pi_2 \U(t,s)\begin{pmatrix} 0\\x\end{pmatrix}.\]
Since the right-hand side is continuous in $(t,s)$, we have that $(t,s)\mapsto \ddd{t}S(t,s)x$ is continuous.
Moreover,
\[\left.\ddd{t}S(t,s)x\right|_{t=s} = \pi_2 \U(s,s)\begin{pmatrix} 0\\x\end{pmatrix} = x.\]
 
 \item
 We have already established that $\ddd{s}S(t,s)x$ exists for all $x\in Z$ and $(t,s)\in\Delta$, and that
 \[\ddd{s} S(t,s)x = -\pi_1 \U(t,s)\begin{pmatrix}x\\0\end{pmatrix}.\]
 Since the right-hand side is continuous in $s$, $s\mapsto S(t,s)x$ is continuously differentiable with
 \[\ddd{s}S(t,s)x = -\pi_1 \U(t,s)\begin{pmatrix} x\\0\end{pmatrix}.\]
 Since the right-hand side is continuous in $(t,s)$, we have that $(t,s)\mapsto \ddd{s}S(t,s)x$ is continuous. 
 Moreover,
 \[\left. \ddd{s} S(t,s)x\right|_{t=s} = -\pi_1 \U(s,s)\begin{pmatrix}x\\0\end{pmatrix} = -x.\]
\end{abc}

\medskip

\textbf{(S2)}: We have already established $S(t,s)Z\subseteq D$, so since $D\subseteq Z$ we have $S(t,s)D\subseteq D$ for all $(t,s)\in\Delta$.

Let $x\in D$. 
By \textbf{(U4)} and the formulas at the beginning of the proof we have that $\Delta\ni(t,s)\mapsto \U(t,s)\begin{pmatrix}0\\x\end{pmatrix}$ is twice continuously differentiable. Thus, also $\Delta\ni(t,s)\mapsto S(t,s)x = \pi_1 \U(t,s)\begin{pmatrix}0\\x\end{pmatrix}$ is twice continuously differentiable. For the second derivatives, with \textbf{(U4)} we compute:

\begin{abc}
 \item
 \begin{align*}
    \frac{\partial^2}{\partial{t^2}}S(t,s)x & = \frac{\partial^2}{\partial{t^2}} \pi_1 \U(t,s) \begin{pmatrix} 0\\x\end{pmatrix} =  \frac{\partial}{\partial{t}}\pi_1 \frac{\partial}{\partial{t}} \U(t,s) \begin{pmatrix} 0\\x\end{pmatrix} = \frac{\partial}{\partial{t}}\pi_1 \A(t)\U(t,s) \begin{pmatrix} 0\\x\end{pmatrix} \\
    & = \frac{\partial}{\partial{t}}\pi_2 \U(t,s)\begin{pmatrix} 0\\x\end{pmatrix}
    = \pi_2 \frac{\partial}{\partial{t}} \U(t,s)\begin{pmatrix} 0\\x\end{pmatrix}
    = \pi_2 \A(t)\U(t,s)\begin{pmatrix} 0\\x\end{pmatrix} = A(t)S(t,s)x.
\end{align*}
 \item
\begin{align*}
    \frac{\partial^2}{\partial{s^2}}S(t,s)x & = \pi_1 \frac{\partial^2}{\partial{s^2}}\U(t,s) \begin{pmatrix} 0\\x\end{pmatrix} =  -\pi_1 \frac{\partial}{\partial{s}} \U(t,s) \A(s)\begin{pmatrix} 0\\x\end{pmatrix} = -\pi_1 \frac{\partial}{\partial{s}} \U(t,s)\begin{pmatrix} x\\0\end{pmatrix} \\
    & = \pi_1 \U(t,s)\A(s)\begin{pmatrix} x\\0\end{pmatrix}
    = \pi_1 \U(t,s)\begin{pmatrix} 0\\A(s)x\end{pmatrix} = S(t,s)A(s)x.
\end{align*} 
 \item
\begin{align*}
    \frac{\partial}{\partial{s}}\frac{\partial}{\partial{t}}S(t,s)x & = \frac{\partial}{\partial{s}} \pi_1 \A(t)\U(t,s) \begin{pmatrix} 0\\x\end{pmatrix} = -\pi_1 \A(t) \U(t,s)\A(s)\begin{pmatrix} 0\\x\end{pmatrix} = -\pi_2 \U(t,s) \begin{pmatrix} x\\0\end{pmatrix}
\end{align*}
This also yields the last assertion by \textbf{(U1)}. 
\end{abc}

\medskip

\textbf{(S3)}: 
Let $(t,s)\in\Delta$ and $x\in D$. Then by \textbf{(U4)} we have $\frac{\partial}{\partial{s}}S(t,s)x = -\pi_1 \U(t,s) \begin{pmatrix} x\\0\end{pmatrix}\in D$ by \textbf{(U3)}. 

Moreover, by \textbf{(U4)} we observe that $\frac{\partial^2}{\partial{t^2}} \frac{\partial}{\partial{s}}S(t,s)x$ and $\frac{\partial^2}{\partial{s^2}} \frac{\partial}{\partial{t}}S(t,s)x$ exist and we can compute these derivatives:

\begin{abc}
 \item
 \begin{align*}
    \frac{\partial^2}{\partial{t^2}} \frac{\partial}{\partial{s}}S(t,s)x & = - \frac{\partial^2}{\partial{t^2}} \pi_1 \U(t,s) \begin{pmatrix} x\\0\end{pmatrix} = -\frac{\partial}{\partial{t}} \pi_1 \A(t) \U(t,s) \begin{pmatrix} x\\0\end{pmatrix}
    = -\frac{\partial}{\partial{t}} \pi_2 \U(t,s) \begin{pmatrix} x\\0\end{pmatrix} \\
    & = -\pi_2 \A(t) \U(t,s) \begin{pmatrix} x\\0\end{pmatrix} = A(t) \Bigl(-\pi_1 \U(t,s)\begin{pmatrix} x\\0\end{pmatrix}\Bigr) = A(t) \frac{\partial}{\partial{s}}S(t,s)x.
\end{align*}
 \item
 \begin{align*}
    \frac{\partial^2}{\partial{s^2}} \frac{\partial}{\partial{t}}S(t,s)x & = \frac{\partial^2}{\partial{s^2}} \pi_2 \U(t,s) \begin{pmatrix} 0\\x\end{pmatrix} = -\frac{\partial}{\partial{s}} \pi_2 \U(t,s)\A(s) \begin{pmatrix} 0\\x\end{pmatrix}
    = -\frac{\partial}{\partial{s}} \pi_2 \U(t,s) \begin{pmatrix} x\\0\end{pmatrix} \\
    & = \pi_2 \U(t,s) \A(s)\begin{pmatrix} x\\0\end{pmatrix} = \pi_2 \U(t,s) \begin{pmatrix} 0\\A(s)x\end{pmatrix} = \frac{\partial}{\partial{t}}S(t,s)A(s)x,
\end{align*}
where the last equality follows from \textbf{(S1)}(c).
 \item
The continuity of $\Delta\ni (t,s)\mapsto A(t) \frac{\partial}{\partial{s}}S(t,s)x$ follows from strong continuity of $t\mapsto A(t)y$ for all $y\in D$, continuity of $(t,s)\mapsto \ddd{s}S(t,s)x$ for all $x\in D$ and the fact that $\ddd{s}S(t,s)D\subseteq D$.
\end{abc}

\medskip

\textbf{(S4)}: 
Let $(t,s),(s,r)\in\Delta$ and $x\in D$. Then, by \textbf{(U1)}, we compute
	\begin{align*}
	\left(-\frac{\partial}{\partial{s}}S(t,s)\right)S(s,r)x+S(t,s)\frac{\partial}{\partial{s}}S(s,r)x 
	& = \pi_1 \U(t,s) \begin{pmatrix} \pi_1\\\pi_2\end{pmatrix} \U(s,r)\begin{pmatrix}0\\x\end{pmatrix}\\
	& = \pi_1 \U(t,s)\U(s,r) \begin{pmatrix}0\\x\end{pmatrix} = \pi_1 \U(t,r)\begin{pmatrix}0\\x\end{pmatrix} = S(t,r)x.
	\end{align*}

\bigskip

Let us prove the converse implication \eqref{thm:Gen:item:nACP2}$\Longrightarrow$\eqref{thm:Gen:item:nACP}.

\medskip

For $(t,s)\in\Delta$ and $\begin{pmatrix}x\\y\end{pmatrix}\in D\times X$ we define
\begin{align*}
\U(t,s)\begin{pmatrix}x\\y\end{pmatrix}: = \begin{pmatrix}
-\frac{\partial}{\partial s}S(t,s)x+ S(t,s)y\\
-\frac{\partial}{\partial t}\frac{\partial}{\partial s}S(t,s)x + \frac{\partial}{\partial t}S(t,s)y
\end{pmatrix}=\begin{pmatrix}
-\frac{\partial}{\partial s}S(t,s) & S(t,s)\\
-\frac{\partial}{\partial t}\frac{\partial}{\partial s}S(t,s) & \frac{\partial}{\partial t}S(t,s)
\end{pmatrix}\begin{pmatrix}x\\y\end{pmatrix}\in Z\times X.
\end{align*}

We first show that $(\U(t,s))_{(t,s)\in\Delta}$ can be extended to a family in $\L(\Z)$.
Let $x\in Z$. Then there exists $(x_n)$ in $D$ such that $x_n\to x$ in $Z$. Since $S(t,s)\in \L(X,Z)$ by Lemma \ref{lem:continuity_interpolationspaces}(a), by \textbf{(S2)} and $D\subseteq Z$ we have that the functions $s\mapsto S(t,s)x_n\in Z$ converge pointwise to $s\mapsto S(t,s)x\in Z$. Moreover, we have
\[\ddd{s}S(t,s)x_n = \ddd{s}S(s,s)x_n + \int_s^t \ddd{t}\ddd{s}S(\tau,s)x_n\,\dd \tau = -x_n + \int_s^t \ddd{t}\ddd{s}S(\tau,s)x_n\,\dd \tau\]
for all $n\in\NN$, by \textbf{(S1)}(c). By assumption, we have that the functions $t\mapsto \ddd{t}\ddd{s} S(t,s)x_n$ converge uniformly to $t\mapsto \ddd{t}\ddd{s}S(t,s)x$. Thus, also the functions $t\mapsto \ddd{s}S(t,s)x_n\in Z$ converge uniformly. Thus, $t\mapsto S(t,s)x$ is differentiable. Lemma \ref{lem:continuity_interpolationspaces}(b) now yields $\ddd{s}S(t,s)\in\L(Z)$ for all $(t,s)\in\Delta$.
As $\ddd{t}\ddd{s}S(t,s)\in \L(Z,X)$ for all $(t,s)\in\Delta$ by assumption, we can continuously extend $\U(t,s)$ to $\Z$ for all $(t,s)\in\Delta$.

\medskip

We now check \textbf{(U1)}--\textbf{(U4)}.

\medskip

\textbf{(U1)}:
By \textbf{(S1)}(c)--(d) and \textbf{(S2)}(c), for $\begin{pmatrix}x\\y\end{pmatrix}\in D\times X$ we obtain $\U(t,t)\begin{pmatrix}x\\y\end{pmatrix} = \begin{pmatrix}x\\y\end{pmatrix}$ for all $t\in[0,T]$. Thus, by continuity and denseness of $D$ in $Z$, $\U(t,t)=\Id$ for all $t\in[0,T]$.

Let $\begin{pmatrix}x\\y\end{pmatrix}\in D\times D$ and $(t,s), (s,r)\in\Delta$. By \textbf{(S3)} we have $\ddd{r}S(s,r)x\in D\subseteq Z$, and $S(s,r)y\in Z$ by \textbf{(S2)}. Thus, by \textbf{(S4)}, we observe
%\todo{Kommentar: Die Klammern bei den zweiten Ableitungen sind zwar formal falsch, aber irgendwie noetig, um die Doppelnutzung der Variable $s$ zu trennen.}
\begin{align*}
&\U(t,s)\U(s,r)\begin{pmatrix}x\\y\end{pmatrix}=\begin{pmatrix}
-\frac{\partial}{\partial s}S(t,s) & S(t,s)\\
-\frac{\partial}{\partial t}\frac{\partial}{\partial s}S(t,s) & \frac{\partial}{\partial t}S(t,s)
\end{pmatrix}\begin{pmatrix}
-\frac{\partial}{\partial r}S(s,r)x + S(s,r)y\\
-\frac{\partial}{\partial s}\frac{\partial}{\partial r}S(s,r)x + \frac{\partial}{\partial s}S(s,r)y
\end{pmatrix}\\
=&\begin{pmatrix}
\left(\ddd{s}S(t,s)\right)\ddd{r}S(s,r)x-S(t,s)\ddd{s}\ddd{r}S(s,r)x + \left(-\ddd{s}S(t,s)\right)S(s,r)y+S(t,s)\ddd{s}S(s,r)y\\
\left(\frac{\partial}{\partial t}\frac{\partial}{\partial s}S(t,s)\right)\frac{\partial}{\partial r}S(s,r)x-\frac{\partial}{\partial t}S(t,s)\frac{\partial}{\partial s}\frac{\partial}{\partial r}S(s,r)x -\left(\frac{\partial}{\partial t}\frac{\partial}{\partial s}S(t,s)\right)S(s,r)y+\frac{\partial}{\partial t}S(t,s)\frac{\partial}{\partial s}S(s,r)y
\end{pmatrix}\\
=&\begin{pmatrix}
-\ddd{r}\left(\left(-\ddd{s}S(t,s)\right)S(s,r)x+S(t,s)\ddd{s}S(s,r)x\right) + \left(-\ddd{s}S(t,s)\right)S(s,r)y+S(t,s)\ddd{s}S(s,r)y\\
-\ddd{t}\ddd{r}\left(\left(-\ddd{s}S(t,s)\right)S(s,r)x+S(t,s)\ddd{s}S(s,r)x\right) + \ddd{t}\left(\left(-\ddd{s}S(t,s)\right)S(s,r)+S(t,s)\ddd{s}S(s,r)\right)y
\end{pmatrix}\\
=&\begin{pmatrix}
-\frac{\partial}{\partial r}S(t,r)x + S(t,r)y\\
-\frac{\partial}{\partial t}\frac{\partial}{\partial r}S(t,r)x + \frac{\partial}{\partial t}S(t,r)y
\end{pmatrix}=\U(t,r)\begin{pmatrix}x\\y\end{pmatrix}.
\end{align*}
Since $D\times D$ is dense in $\Z$ and $\U(t,r)$, $\U(t,s)$ and $\U(s,r)$ are continuous, we obtain the assertion.

\medskip

\textbf{(U2)}:
We show componentwise the strong continuity of $(\mathcal{U}(t,s))_{(t,s)\in\Delta}$. To do so for the first component, let $\begin{pmatrix}x\\y\end{pmatrix}\in Z\times X$, then
\[
\pi_1\mathcal{U}(t,s)\begin{pmatrix}x\\y\end{pmatrix}=\frac{\partial}{\partial{s}}S(t,s)x+S(t,s)y \in Z.
\]
By assumption on $(S(t,s))_{(t,s)\in\Delta}$, we obtain strong continuity in the first component. 

Likewise, for $\begin{pmatrix}x\\y\end{pmatrix}\in D\times X$ we observe that
\[
\pi_2\mathcal{U}(t,s)\binom{x}{y}=-\frac{\partial}{\partial t}\frac{\partial}{\partial{s}}S(t,s)x+\frac{\partial}{\partial t}S(t,s)y.
\]
By \textbf{(S2)} we know that $(t,s)\mapsto S(t,s)x$ is twice continuously differentiable. As we also assumed that $\|\ddd{t}\ddd{s} S(t,s)x\|_X\leq C\|x\|_{Z}$ for all $x\in Z$ and $(t,s)\in\Delta$ we also conclude the strong continuity of the first term.
For the second term, \textbf{(S1)}(c) directly yields the strong continuity. Thus, we obtain strong continuity in the second component.

\medskip

\textbf{(U3)}:
Let $\begin{pmatrix} x\\y\end{pmatrix}\in \D$ and $(t,s)\in \Delta$. Then $x\in D$ and therefore $\frac{\partial}{\partial s} S(t,s)x\in D$ by \textbf{(S3)}. Moreover, $S(t,s)y\in D$ by assumption, and $\ddd{t}\ddd{s}S(t,s)x\in Z$ and $\ddd{t}S(t,s)y\in Z$ by assumption. Hence,
\[\U(t,s)\begin{pmatrix} x\\y\end{pmatrix} = \begin{pmatrix} \frac{\partial}{\partial s} S(t,s)x + S(t,s)y\\-\frac{\partial}{\partial t} \frac{\partial}{\partial s} S(t,s)x+\frac{\partial}{\partial t}S(t,s)y\end{pmatrix}\in \D.\]

\medskip

\textbf{(U4)}:
Let $\begin{pmatrix}x\\y\end{pmatrix}\in \D$. Then $\Delta\ni(t,s)\mapsto S(t,s)x$ is two times continuously differentiable by \textbf{(S2)}, and $\frac{\partial^2}{\partial{t^2}}\frac{\partial}{\partial{s}}S(t,s)x$ exists by \textbf{(S3)}. 
Moreover, by assumption $\frac{\partial^2}{\partial{t^2}}S(t,s)y$ exists and $\frac{\partial^2}{\partial{t^2}}S(t,s)y = A(t)S(t,s)y$.
Thus, we observe that
\[\frac{\partial}{\partial{t}}\U(t,s)\begin{pmatrix} x\\y\end{pmatrix} = \begin{pmatrix}
 -\frac{\partial}{\partial{t}}\frac{\partial}{\partial{s}}S(t,s)x + \frac{\partial}{\partial{t}}S(t,s)y\\
 -\frac{\partial^2}{\partial{t^2}}\frac{\partial}{\partial{s}}S(t,s)x + \frac{\partial^2}{\partial{t^2}}S(t,s)y\end{pmatrix}
 \]
exists, and by assumption and \textbf{(S3)}(a) we conclude
\begin{align*}
    \frac{\partial}{\partial{t}}\U(t,s)\begin{pmatrix} x\\y\end{pmatrix} & = \begin{pmatrix}
 -\frac{\partial}{\partial{t}}\frac{\partial}{\partial{s}}S(t,s)x + \frac{\partial}{\partial{t}}S(t,s)y\\
 -\frac{\partial^2}{\partial{t^2}}\frac{\partial}{\partial{s}}S(t,s)x + \frac{\partial^2}{\partial{t^2}}S(t,s)y\end{pmatrix} = \begin{pmatrix}
 -\frac{\partial}{\partial{t}}\frac{\partial}{\partial{s}}S(t,s)x + \frac{\partial}{\partial{t}}S(t,s)y\\
 -A(t)\frac{\partial}{\partial{s}}S(t,s)x + A(t)S(t,s)y\end{pmatrix} \\
 & = \begin{pmatrix} 0 & \Id\\ A(t) & 0 \end{pmatrix}\U(t,s)\begin{pmatrix} x\\y\end{pmatrix} = \A(t)\U(t,s)\begin{pmatrix} x\\y\end{pmatrix}.\end{align*} 
 
To show the other statement, we first let $\begin{pmatrix}x\\y\end{pmatrix}\in D\times D$. Then $\Delta\ni(t,s)\mapsto S(t,s)x$ and $\Delta\ni(t,s)\mapsto S(t,s)y$ are two times continuously differentiable by \textbf{(S2)}. Moreover, $\frac{\partial^2}{\partial{s^2}}\frac{\partial}{\partial{t}}S(t,s)x$ exists by \textbf{(S3)}. 
Thus,
\[\frac{\partial}{\partial{s}}\U(t,s)\begin{pmatrix} x\\y\end{pmatrix} = \begin{pmatrix}
 -\frac{\partial^2}{\partial{s^2}}S(t,s)x + \frac{\partial}{\partial{s}}S(t,s)y\\
 -\frac{\partial^2}{\partial{s^2}}\frac{\partial}{\partial{t}}S(t,s)x + \frac{\partial}{\partial{s}}\frac{\partial}{\partial{t}}S(t,s)y\end{pmatrix}
 \]
exists, and by \textbf{(S2)}(b) and \textbf{(S3)}(b) we conclude
\begin{align*}
    \frac{\partial}{\partial{s}}\U(t,s)\begin{pmatrix} x\\y\end{pmatrix} & = \begin{pmatrix}
 -\frac{\partial^2}{\partial{s^2}}S(t,s)x + \frac{\partial}{\partial{s}}S(t,s)y\\
 -\frac{\partial^2}{\partial{s^2}}\frac{\partial}{\partial{t}}S(t,s)x + \frac{\partial}{\partial{s}}\frac{\partial}{\partial{t}}S(t,s)y\end{pmatrix} = \begin{pmatrix}
 -S(t,s)A(s)x + \frac{\partial}{\partial{s}}S(t,s)y\\
 -\frac{\partial}{\partial{t}}S(t,s)A(s)x + \frac{\partial}{\partial{s}}\frac{\partial}{\partial{t}}S(t,s)y\end{pmatrix}\\
 & = -\U(t,s) \begin{pmatrix} 0 & \Id\\ A(s) & 0 \end{pmatrix} \begin{pmatrix} x\\y\end{pmatrix} = -\U(t,s)\A(s)\begin{pmatrix} x\\y\end{pmatrix}.
\end{align*}

Now, let $\begin{pmatrix}x\\y\end{pmatrix}\in \D$. There exists $(y_n)$ in $D$ such that $y_n\to y$ in $Z$.
Then the sequence of continuous functions $s\mapsto \U(t,s)\begin{pmatrix}x\\y_n\end{pmatrix}$ converge uniformly to $s\mapsto \U(t,s)\begin{pmatrix}x\\y\end{pmatrix}$ by boundedness of $(\U(t,s))_{(t,s)\in\Delta}$. Moreover, the sequence of functions $t\mapsto \ddd{s}\U(t,s)\begin{pmatrix}x\\y_n\end{pmatrix} = \U(t,s)\A(s)\begin{pmatrix} x\\y_n\end{pmatrix} = \U(t,s)\begin{pmatrix} y_n\\A(s)x\end{pmatrix}$ converges uniformly to $s\mapsto \U(t,s)\A(s)\begin{pmatrix}x\\y\end{pmatrix}$, since $(\U(t,s))_{(t,s)\in\Delta}$ is bounded and $s\mapsto A(s)x$ is continuous. 
Thus, 
$\ddd{s}\U(t,s)\begin{pmatrix}x\\y\end{pmatrix}$ exists and
\[\ddd{s}\U(t,s)\begin{pmatrix}x\\y\end{pmatrix} = \U(t,s)\A(s)\begin{pmatrix}x\\y\end{pmatrix}.\]
\end{proof}

\begin{remark}
    If Theorem \ref{thm:Gen}(\ref{thm:Gen:item:nACP2}) is true the statement in Proposition \ref{prop:solution_nACP2} can be extended to $y\in Z$.
\end{remark}

\begin{remark}[The Choice of the Space $Z$]
    Let us remark on the choice of the space $Z$.
    \begin{abc}
     \item 
     Let us consider the autonomous case, i.e., let $A$ be a densely defined closed operator and $A(t) = A$ for all $t\in[0,T]$.
     Then $Z$ is the Kisy\'{n}ski-space given by
     \[Z = \{x\in X:\ S(\cdot,\cdot)x\in C(\Delta;D)\},\]
     cf.\ \cite{Ki1972} or \cite[Theorem 3.14.11]{ABHN2011}. Note that in this case the space is uniquely defined.
     In this case, $A$ generates a so-called cosine family.
     
     \item
     For all $t\in [0,T]$ let $A(t)$ generate a cosine family. Then by \cite{TW1981}, w.l.o.g.\ we may assume that there exists $(B(t))_{t\in[0,T]}$ such that $B(t)^2 = A(t)$ for all $t\in [0,T]$. Let us assume that $(B(t))_{t\in[0,T]}$ satisfies Assumption \ref{assump:ConstDom} and Assumption \ref{ass:EquiNorm}. Then, under some further assumptions, the space $Z$ can be chosen to be $Z:=\dom(B(t))$ for some/all $t\in[0,T]$ (equipped with the graph norm); cf.\ \cite{K1994,K1995,K1995Fund}.
    \end{abc}
\end{remark}

As we have seen in the above remark, in the autonomous case we have an explicit description of the space $Z$.

\begin{conjecture}
Let $(S(t,s))_{(t,s)\in\Delta}$ be an evolutionary fundamental solution on $X$ of \eqref{eqn:nACP2} associated to $(A(t),\dom(A(t)))_{t\in [0,T]}$. We conjecture that $Z$ has the form
\[Z = \{x\in X:\ S(\cdot,\cdot)x\in \mathrm{C}(\Delta;D)\}\]
equipped with the norm $\|\cdot\|_Z$ given by
\[\|x\|_Z := \|x\|_X + \sup_{(t,s)\in\Delta} \|A(t)S(t,s)x\|_X.\]
\end{conjecture}

\section{A Bounded Perturbation Type Result}
\label{sec:BoundPert}

Let $(B(t))_{t\in\left[0,T\right]}$ be a family of bounded operators on $X$. We first prove a first-order non-autonomous bounded perturbation type result similar to \cite[Chapter VI, Cor.~9.20]{EN} which fits into our framework of fundamental solutions. %However, we have to take into account that Theorem \ref{thm:Gen} also make use of systems.

\begin{proposition}\label{prop:BoundPertEvoFamAmend}
Let $(A(t),\dom(A(t)))_{t\in\left[0,T\right]}$ be a family of densely defined closed operators on a Banach space $X$ satisfying Assumption \ref{assump:ConstDom}, Assumption \ref{ass:EquiNorm} and Assumption \ref{ass:regularity_A}. Let $(S(t,s))_{(t,s)\in\Delta}$ be an evolutionary fundamental solution of the corresponding \eqref{eqn:nACP} on $X$ and let $Z$ be the space occuring in Theorem \ref{thm:Gen}
such that
	for all $(t,s)\in\Delta$ we have 
	\begin{itemize}
	\item $S(t,s)X\subseteq Z$, $S(t,s)Z\subseteq D$, $(t,s)\mapsto S(t,s)x\in Z$ is continuous for all $x\in X$, 
	\item $\ddd{t}S(t,s)Z\subseteq Z$, $\frac{\partial^2}{\partial t^2}S(t,s)x$ exists for all $x\in Z$ and $\frac{\partial^2}{\partial t^2}S(t,s)x = A(t)S(t,s)x$, 
	\item $\ddd{s} S(t,s)x$ exists for all $x\in Z$, $\ddd{s} S(t,s)Z\subseteq Z$ and $(t,s)\mapsto \ddd{s}S(t,s)x\in Z$ is continuous for all $x\in Z$, 
	\item $\ddd{t}\ddd{s}S(t,s)D\subseteq Z$, $\ddd{t}\ddd{s} S(t,s)x$ exists for all $x\in Z$ and there exists $C\geq 0$ such that $\|\ddd{t}\ddd{s} S(t,s)x\|_X\leq C\|x\|_{Z}$ for all $x\in Z$ and $(t,s)\in\Delta$.
	\end{itemize}
Let $(\U(t,s))_{(t,s)\in\Delta}$ be a fundamental solution of \eqref{eqn:nACP} on $\mathcal{Z}:=Z\times X$ corresponding to $(\A(t),\dom(\A(t)))_{t\in\left[0,T\right]}$. Let $B(\cdot)\in\mathrm{C}\left(\left[0,T\right];\LLL_\mathrm{s}(Z)\right)$. Then there exists a fundamental solution $(\V(t,s))_{(t,s)\in\Delta}$ of \eqref{eqn:nACP} on $\mathcal{Z}$ corresponding to the family of operators $(\A(t)+\B(t),\dom(\A(t)))_{t\in\left[0,T\right]}$, where
\[
\B(t):=\begin{pmatrix}
0&0\\
B(t)&0
\end{pmatrix},\quad t\in\left[0,T\right].
\]
\end{proposition}                                                                             

\begin{proof}
Firstly, by Theorem \ref{thm:Gen} we know that $(\U(t,s))_{(t,s)\in\Delta}$ is indeed a fundamental solution of \eqref{eqn:nACP} on $\Z=Z\times X$. By \cite[Chapter VI, Cor.~9.20]{EN} there exists a family of operators $(\V(t,s))_{(t,s)\in\Delta}$ satisfying \textbf{(U1)} and \textbf{(U2)}. Moreover, we know that the variation of constants formula holds, i.e., one has that
\begin{align}\label{eqn:VarConstBoundPert}
\begin{split}
\V(t,s)\binom{z}{x}&=\U(t,s)\binom{z}{x}+\int_s^t{\U(t,r)\B(r)\V(r,s)\binom{z}{x}\ \dd{r}}\\
&=\U(t,s)\binom{z}{x}+\int_s^t{\V(t,r)\B(r)\U(r,s)\binom{z}{x}\ \dd{r}},
\end{split}
\end{align}
for all $z\in Z$ and $x\in X$. In order to show that $(\V(t,s))_{(t,s)\in\Delta}$ is a fundamental solution according to Definition \ref{def:FundSolnACP} we have to show that also \textbf{(U3)} and \textbf{(U4)} hold. To do so, let $\binom{x}{y}\in\mathcal{D}:=D\times Z$ be arbitrary. Then $\U(t,s)\binom{x}{y}\in\mathcal{D}$ by the assumption that $(\U(t,s))_{(t,s)\in\Delta}$ is a fundamental solution, cf. Definition \ref{def:FundSolnACP}. As $(\V(t,s))_{(t,s)\in\Delta}$ is a family of operators in $Z\times X$ we obviously have $\V(t,s)(D\times X)\subseteq Z\times X$. Now, we observe that by the assumption that $B(\cdot)\in\mathrm{C}\left(\left[0,T\right];\LLL_\mathrm{s}(Z)\right)$ and the explicit representation of the operators we have $\B(t)(Z\times X)\subseteq\left\{0\right\}\times Z\subseteq D\times Z$ so that the integral term appearing in the variation of constant formula \eqref{eqn:VarConstBoundPert} is in $D\times Z$ as well. This shows that \textbf{(U3)} holds. For \textbf{(U4)}, let $\binom{x}{y}\in\mathcal{D}$. Then we have
\begin{align*}
\frac{\partial}{\partial{t}}\V(t,s)\binom{x}{y}&=\frac{\partial}{\partial{t}}\U(t,s)\binom{x}{y}+\frac{\partial}{\partial{t}}\int_s^t{\U(t,r)\B(r)\V(r,s)\binom{x}{y}\ \dd{r}}\\
&=\A(t)\U(t,s)\binom{x}{y}+\U(t,t)\B(t)\V(t,s)\binom{x}{y}+\int_s^t{\frac{\partial}{\partial{t}}\U(t,r)\B(r)\V(r,s)\binom{x}{y}\ \dd{r}}\\
%&=\A(t)\U(t,s)x+\B(t)\V(t,s)x+\int_s^t{\frac{\partial}{\partial{t}}U(t,\theta)U(\theta,r)B(r)V(r,s)x\ \dd{r}}\\
&=\A(t)\U(t,s)\binom{x}{y}+\B(t)\V(t,s)\binom{x}{y}+\int_s^t{\A(t)\U(t,r)\B(r)\V(r,s)\binom{x}{y}\ \dd{r}}\\
&=\A(t)\U(t,s)\binom{x}{y}+\B(t)\V(t,s)\binom{x}{y}+\A(t)\int_s^t{\U(t,r)\B(r)\V(r,s)\binom{x}{y}\ \dd{r}}\\
%&=\A(t)\U(t,s)x+\B(t)\V(t,s)x+\A(t)\int_s^t{\U(t,r)\B(r)\V(r,s)x\ \dd{r}}\\
%&=\A(t)\U(t,s)\binom{x}{y}+\B(t)\V(t,s)\binom{x}{y}+\A(t)\left[\V(t,s)\binom{x}{y}-\U(t,s)\binom{x}{y}\right]\\
&=(\A(t)+\B(t))\V(t,s)\binom{x}{y},
\end{align*}
where we have used Hille's theorem. Moreover,
\begin{align*}
\frac{\partial}{\partial{s}}\V(t,s)\binom{x}{y}&=\frac{\partial}{\partial{s}}\U(t,s)\binom{x}{y}+\frac{\partial}{\partial{s}}\int_s^t{\V(t,r)\B(r)\U(r,s)\binom{x}{y}\ \dd{r}}\\
&=-\U(t,s)\A(s)\binom{x}{y}-\V(t,s)\B(s)\U(s,s)\binom{x}{y}+\int_s^t{\frac{\partial}{\partial{s}}\V(t,r)\B(r)\U(r,s)\binom{x}{y}\ \dd{r}}\\
&=-\U(t,s)\A(s)\binom{x}{y}-\V(t,s)\B(s)\binom{x}{y}+\int_s^t{\V(t,r)\B(r)\frac{\partial}{\partial{t}}\U(r,s)\binom{x}{y}\ \dd{r}}\\
&=-\U(t,s)\A(s)\binom{x}{y}-\V(t,s)\B(s)\binom{x}{y}-\int_s^t{\V(t,r)\B(r)\U(r,s)\A(s)\binom{x}{y}\ \dd{r}}\\
%&=-\U(t,s)\A(s)\binom{x}{y}-\V(t,s)\B(s)\binom{x}{y}-\left[\V(t,s)-\U(t,s)\right]\A(s)\binom{x}{y}\\
&=-\V(t,s)(\A(s)+\B(s))\binom{x}{y}.
\end{align*}
Hence, we conclude that there exists a fundamental solution of \eqref{eqn:nACP} corresponding to the family of operators $(\A(t)+\B(t),\dom(\A(t)))_{t\in\left[0,T\right]}$ according to Definition \ref{def:FundSolnACP}.
\end{proof}

Now, by using Theorem \ref{thm:Gen} in combination with Proposition \ref{prop:BoundPertEvoFamAmend} we obtain the following result.

\begin{theorem}
\label{thm:bounded_perturbation}
Let $(A(t),\dom(A(t)))_{t\in\left[0,T\right]}$ be a family of densely defined closed operators on a Banach space $X$ satisfying Assumption \ref{assump:ConstDom}, Assumption \ref{ass:EquiNorm} and Assumption \ref{ass:regularity_A}.
Let $(S(t,s))_{(t,s)\in\Delta}$ be an evolutionary fundamental solution of \eqref{eqn:nACP2} 
such that
	for all $(t,s)\in\Delta$ we have 
	\begin{itemize}
	\item $S(t,s)X\subseteq Z$, $S(t,s)Z\subseteq D$, $(t,s)\mapsto S(t,s)x\in Z$ is continuous for all $x\in X$, 
	\item $\ddd{t}S(t,s)Z\subseteq Z$, $\frac{\partial^2}{\partial t^2}S(t,s)x$ exists for all $x\in Z$ and $\frac{\partial^2}{\partial t^2}S(t,s)x = A(t)S(t,s)x$, 
	\item $\ddd{s} S(t,s)x$ exists for all $x\in Z$, $\ddd{s} S(t,s)Z\subseteq Z$ and $(t,s)\mapsto \ddd{s}S(t,s)x\in Z$ is continuous for all $x\in Z$, 
	\item $\ddd{t}\ddd{s}S(t,s)D\subseteq Z$, $\ddd{t}\ddd{s} S(t,s)x$ exists for all $x\in Z$ and there exists $C\geq 0$ such that $\|\ddd{t}\ddd{s} S(t,s)x\|_X\leq C\|x\|_{Z}$ for all $x\in Z$ and $(t,s)\in\Delta$.
	\end{itemize}
Let $B(\cdot)\in\mathrm{C}\left(\left[0,T\right];\LLL_\mathrm{s}(X)\right)\cap\mathrm{C}\left(\left[0,T\right];\LLL_\mathrm{s}(Z)\right)$. Then there exists an evolutionary fundamental solution of \eqref{eqn:nACP2} on $X$ associated to a family of operator $(A(t)+B(t),\dom(A(t)))_{t\in\left[0,T\right]}$.
\end{theorem}

\begin{proof}
By the assumptions on the evolutionary fundamental solution $(S(t,s))_{(t,s)\in\Delta}$ and Theorem \ref{thm:Gen}, there exists a fundamental solution $(\mathcal{U}(t,s))_{(t,s)\in\Delta}$ of \eqref{eqn:nACP} on $\mathcal{Z}:=Z\times X$. For $t\in\left[0,T\right]$ we define a family of operators on $\Z=Z\times X$ by
\begin{align*}
\mathcal{B}(t):=\begin{pmatrix}
0&0\\
B(t)&0
\end{pmatrix}.
\end{align*}
By Proposition \ref{prop:BoundPertEvoFamAmend} we obtain a fundamental solution $(\V(t,s))_{(t,s)\in\Delta}$ of \eqref{eqn:nACP} on $\mathcal{Z}$ corresponding to the family of operators $(\A(t)+\B(t),\dom(\A(t)))_{t\in\left[0,T\right]}$. By using Theorem \ref{thm:Gen} again we conclude the result.
%We observe, that by our assumption that $B(\cdot)\in\mathrm{C}\left(\left[0,T\right],\mathcal{L}_\mathrm{s}(Z)\right)$ one directly obtains $\mathcal{B}(\cdot)\in\mathrm{C}\left(\left[0,T\right],\mathcal{L}(Z\times X)\right)$. To apply Proposition \ref{prop:BoundPertEvoFamAmend} we need to show that $\mathcal{U}(t,s)\mathcal{Z}\subseteq\mathcal{D}$ or $\mathcal{U}(t,s)(Z\times X)\subseteq D\times Z$. However, as we know from the proof of Theorem \ref{thm:Gen} how $(\mathcal{U}(t,s))_{(t,s)\in\Delta}$ is actually explicitly constructed we directly conclude this from the assumptions above. Hence, Proposition \ref{prop:BoundPertEvoFamAmend} applies and we may conclude that there exists a fundamental solution of \eqref{eqn:nACP} on $\mathcal{Z}=Z\times X$ associated to the family of operators $(A(t)+B(t),\dom(A(t)))_{t\in\left[0,T\right]}$. By Theorem \ref{thm:Gen} we conclude the result.
\end{proof}

\begin{remark}
We observe that we do not need any additional assumption on $(S(t,s))_{(t,s)\in\Delta}$ for Theorem \ref{thm:bounded_perturbation} besides the ones already appearing in Theorem \ref{thm:Gen}. Hence, perturbation just relies on the continuity assumption on the perturbing operators.
\end{remark}

\section{Example: Non-Autonomous Wave Equation}
\label{sec:example}

Motivated by \cite[Sec.~5]{HP2016}, we consider the following perturbed non-autonomous wave equation on $\Ell^2\left(0,\pi\right)$ given by
\begin{align}\label{eqn:Appl}
\begin{cases}
\frac{\partial^2}{\partial t^2}w(t,\xi)=\alpha(t)\frac{\partial^2}{\partial\xi^2}w(t,\xi)+\beta(t,\xi)w(t,\xi),&\quad t\in\left(0,T\right],\ \xi\in\left(0,\pi\right),\\
w(t,0)=w(t,\pi)=0,&\quad t\in\left(0,T\right],\\
w(0,\xi)=\varphi(\xi),&\\
\frac{\partial}{\partial t}w(0,\xi)=\psi(\xi),&\quad \xi\in\left[0,\pi\right],
\end{cases}
\end{align}
where $T>0$, $\varphi,\psi\in\Ell^2\left(0,\pi\right)$ and $\alpha\from\left[0,T\right]\to\RR$ is continuously differentiable such that $\alpha(t)\geq1$ for all $t\in\left[0,T\right]$. Moreover, we assume that $\beta\in\mathrm{C}^2([0,T]\times\left[0,\pi\right])$. For $t\in\left[0,T\right]$, we define a family of operators $(A(t),\dom(A(t)))_{t\in\left[0,T\right]}$ on $X:=\Ell^2\left(0,\pi\right)$ by $A(t)=\alpha(t)A_0$ with dense domain $\dom(A(t))=\dom(A_0)=:D$, $t\in\left[0,T\right]$, where
\[
A_0f:=f'',\quad \dom(A_0):=\left\{f\in\mathrm{H}^2\left(0,\pi\right):\ f(0)=f(\pi)=0\right\}
\]
is the Dirichlet Laplacian. In particular, Assumption \ref{assump:ConstDom} is satisfied. 
Since $\alpha$ is continuous, there exists $C>0$ such that $\alpha(t)\leq C$ for all $t\in [0,T]$, and therefore by the assumption $\alpha\geq 1$, we observe
\[\| \cdot\|_{A_0} \leq \|\cdot \|_{A(t)} \leq C\|\cdot\|_{A_0}\]
for all $t\in[0,T]$, which implies that Assumption \ref{ass:EquiNorm} holds.
%\[\frac{1}{C} \|\cdot\|_{A(s)} \leq \|_{A(t)} \leq C\|\cdot\|_{A(s)}\]
%for all $s,t\in[0,T]$.
Moreover, as $\alpha$ is continuously differentiable, Assumption \ref{ass:regularity_A} is satisfied as well.

\medskip
Furthermore, we introduce a family of bounded operators $(B(t))_{t\in\left[0,T\right]}$ by
\[
B(t)f:=\beta(t,\cdot)f,\quad t\in\left[0,T\right]
\] 
Then \eqref{eqn:Appl} has an abstract form of a non-autonomous second-order abstract Cauchy problem
\begin{equation}\label{eq:application}
\begin{cases}
\ddot{u}(t)=(A(t)+B(t))u(t),&\quad t\in\left(0,T\right],\\
u(0)=\varphi,&\\
\dot{u}(0)=\psi.&
\end{cases}
\end{equation}
As elaborated by Henr\'{\i}quez and Pozo in \cite{HP2016}, for $n\in\NN$ let $z_n\from[0,\pi]\to\RR$, $z_n(\xi):= \sqrt{\frac{2}{\pi}}\sin(n\xi)$. Then $(z_n)_{n\in\NN}$ is an orthonormal basis of $\Ell^2(0,\pi)$ of eigenfunctions of $A_0$ corresponding to the sequence of eigenvalues $(-n^2)_{n\in\NN}$ and the family of bounded linear operators $(S(t,s))_{(t,s)\in \Delta}$ on $\Ell^2\left(0,\pi\right)$ defined by 
\[
S(t,s)x:=\sum_{n=1}^\infty{r_n(t,s)\left\langle x,z_n\right\rangle z_n},
\]
provides a fundamental solution to the second-order non-autonomous abstract Cauchy problem associated to $(A(t),\dom(A(t)))_{t\in\left[0,T\right]}$, where the functions $r_n$ denote the solution of the initial value problem
\begin{align}\label{eqn:DiffEqFundSolEx}
\begin{cases}
r''(t)+n^2\alpha(t)r(t)=0,&\quad 0\leq s\leq t\leq T,\\
r(s)=0,&\\
r'(s)=1.&
\end{cases}
\end{align}
Note that we have $|r_n(t,s)|\leq \frac{1}{\sqrt{\alpha(s)} n} \leq \frac{1}{n}$, $|\ddd{t} r_n(t,s)|\leq 1$, $|\ddd{s} r_n(t,s)|\leq 1$ and $|\ddd{t}\ddd{s}r_n(t,s)|\leq n$ for all $(t,s)\in\Delta$ and $n\in\NN$; cf. \cite[(5.12)]{HP2016} as well as \cite{NoyoObayaRojo1995}. By spectral theory,
\[D = \{x\in \Ell^2(0,\pi):\; \sum_{n=1}^\infty n^4 \langle x,z_n\rangle ^2 <\infty\}.\]
Let 
\[
Z:=\mathrm{H}_0^1\left(0,\pi\right) = \left\{f\in\mathrm{H}^1\left(0,\pi\right):\ f(0)=f(\pi)=0\right\} = \{x\in \Ell^2(0,\pi):\ \sum_{n=1}^\infty n^2 \langle x,z_n\rangle ^2 <\infty\}.\]
The estimates on the $r_n$ imply that the hypotheses on $(S(t,s))_{(t,s)\in\Delta}$ in Theorem \ref{thm:bounded_perturbation} are satisfied, due to uniform convergence of the series representations. Note that these hypotheses are the same as those in Theorem \ref{thm:Gen}(\ref{thm:Gen:item:nACP2}). 

\medskip

Since $\beta\in \mathrm{C}^2([0,T]\times [0,\pi])$ we easily obtain $B(\cdot)\in\mathrm{C}([0,T];\LLL_{\mathrm{s}}(X))\cap \mathrm{C}([0,T];\LLL_{\mathrm{s}}(Z))$. Thus, Theorem \ref{thm:bounded_perturbation} yields a fundamental solution to \eqref{eq:application}.

\section*{Acknowledgment}
This work is based on the research supported by the National Research Foundation (Grant number: 150417). It is acknowledged that opinions, findings and conclusions or recommendations expressed in any publication generated by this supported research is that of the author(s). The National Research Foundation accepts no liability whatsoever in this regard.

C.S. thanks for a very pleasent stay at the University of the Free State, South Africa, where this work was done.

\end{document}